\documentclass[12pt]{amsart}
\usepackage{amsmath, amscd, amssymb}
\usepackage{tikz}
\usepackage{graphpap, color}
\usepackage[mathscr]{eucal}
\usepackage{mathrsfs}
\usepackage{pstricks}
\usepackage{cancel}
\usepackage[mathscr]{eucal}
\usepackage{verbatim}
\usepackage[all]{xy}
\usepackage{stmaryrd}

\usepackage[margin=2cm]{geometry}
\def\cA{{\cal A}}

\numberwithin{equation}{section}

\def\sO{{\mathscr O}}

\def\sL{{\mathscr L}}

\def\sO{\mathscr{O}}

\newcommand{\CC}{\mathbb{C}}

\newcommand{\PP}{\mathbb{P}}
\newcommand{\QQ}{\mathbb{Q}}

\newcommand{\ZZ}{\mathbb{Z}}

\newcommand{\TT}{\mathbb{T}}

\newcommand{\cal}{\mathcal}

\def\cC{{\cal C}}
\def\cD{{\cal D}}
\def\cE{{\cal E}}
\def\cF{{\cal F}}

\def\cK{{\cal K}}

\def\cM{{\cal M}}

\def\cV{{\cal V}}

\def\cS{{\cal S}}

\def\cY{{\cal Y}}
\def\cZ{{\cal Z}}
\def\cI{{\cal I}}

\def\fC{\mathfrak{C}}
\def\fD{\mathfrak{D}}

\def\fX{\mathfrak{X}}

\def\a{\mathbf a}
\def\Lau#1{\llceil{#1}\rrceil}

\def\lra{\longrightarrow}

\newcommand{\si}{\sigma}

\def\begeq{\begin{equation}}
\def\endeq{\end{equation}}
\def\and{\quad{\rm and}\quad}

\def\and{\quad\text{and}\quad}

\newtheorem{prop}{Proposition}[section]

\newtheorem{theo}[prop]{Theorem}
\newtheorem{lemm}[prop]{Lemma}
\newtheorem{coro}[prop]{Corollary}

\newtheorem{defi}[prop]{Definition}

\newtheorem{defi-prop}[prop]{Definition-Proposition}

\def\ev{\mathrm{ev}}

\def\sO{{\mathscr O}}

\def\beq{\begin{equation}}
\def\eeq{\end{equation}}

\def\vsp{\vskip5pt}

\def\bee{\begin{equation}}
\def\eeq{\end{equation}}

\def\sC{{\mathscr C}}

\def\bd{{\mathbf d}}

\def\ti{\tilde}

\def\coeff#1{\llbracket{#1}\rrbracket}
\def\bigcoeff#1{\big\llbracket{#1}\big\rrbracket}

\def\hb{\hbar}

\def\Gr{\mathbf{Gr}}
\def\e{\mbox{e}}
\title[Double $J$-function of stable quasimap invariants]{Double $J$-function of stable quasimap invariants for complete intersection in Grassmannian}

\author[M.-L.~Li]{Mu-lin Li}
\address{College of Mathematics and Econometrics, Hunan University, China} \email{mulin@hnu.edu.cn}
\begin{document}

\begin{abstract}
Using localization methods on the moduli space of stable quasimaps to Grassmannian, we give explicit formulas of $J$-function and double $J$-function of stable quasimaps for complete intersection in Grassmannian.
\end{abstract}

\maketitle

\section{introduction}

The moduli space of stable quotients was first constructed and studied by Marian, Oprea and Pandharipande \cite{MOP}. Cooper and Zinger \cite{CZ} calculated the $J$-function of stable quotients for projective complete intersections, and proved that it is related to the genus zero Gromov-Witten invariants by mirror map. Later Zinger \cite{Zinger1} calculated the double and the triple $J$-functions of stable quotients for projective complete intersections. Ciocan-Fontanine, Kim and Maulik~\cite{FKM} generalized the definition of stable quotients to arbitrary GIT quotients by the notion of stable quasimaps. Ciocan-Fontanine and Kim proved that the genus zero stable quasimap invariants (including twisted cases) are related to stable map by mirror map in \cite{FK}, and for the higher genus cases of projective complete intersections in \cite{FK3}.  In this paper, using the abelian non-abelian relation between Grassamannian and product of projective spaces as in \cite{BFK}, we analyze the $J$-function and the double $J$-function of stable quasimaps for complete intersection in Grassmannian cases, which should play a key role in the computation of the genus one  stable quasimap invariants for complete intersection in Grassmannian.

Let $\Gr:=\Gr(2,n)$ be the space of all $2$-dimensional subspaces of $\CC^n$. Recall that the cohomology of $\Gr$ has the presentation:
$$
H^*(\Gr)\cong\QQ[e_1,e_2]/\langle h_{n-1},h_n \rangle,
$$
where $e_1,\, e_2$ are the elementary symmetric polynomials of two variables $x_1,x_2$ (the Chern roots of the dual universal bundle), and $h_{i}$ is
the $i$-th complete symmetric polynomial (sums of all monomials of degree $i$) of $x_1,x_2$. Denote by $x=(x_1,x_2)$. Let $\{\gamma^k_j\}$ be a basis of $ H^*(\Gr)$ such that the Poincar\'{e} dual of the diagonal class $[\Delta]$ can be written as
$$
[\Delta]=\sum_{k,j}\gamma^k_j(x)\otimes\gamma^{2(n-2)-k}_j(x)\in H^*(\Gr)\otimes H^*(\Gr),
$$
 where $\gamma^k_j$ is a symmetry polynomial of degree $k$ in $x_1,x_2$. 
 
Denote by $\overline{Q}_{0,2}(\Gr,d)$ the moduli of genus zero  and degree $d$ stable quasimaps to $\Gr$ with two marked points. Let
$$\ev_i:\overline{Q}_{0,2}(\Gr,d)\to \Gr$$
be the evaluation map at the $i$-th marked point, see \cite{FK}. For each $i=1,2$, let $\psi_i$ be the first Chern class of the universal cotangent line bundle for the $i$-th marked point. Let
$$
\pi:\cC\to \overline{Q}_{0,2}(\Gr,d)
$$
be the universal family, and $\sigma_1,\sigma_2$ be the sections given by the marked points. Let
$$
\cS\subset\CC^n\otimes\sO_{\cC}
$$
be the universal bundle. Let $\sL:=\det\cS^{\vee}$.

For $\ell\!\in\!\ZZ^{\ge0}$ and $\ell$-tuple $\a\!=\!(a_1,\ldots,a_{\ell})\!\in\!(\ZZ^{>0})^{\ell}$ of positive integers,
denote
\begin{gather*}
|\a|=\sum_{k=1}^{\ell}a_k,  \qquad
\langle\a\rangle=\prod_{k}\!a_k\,, \qquad
\a!=\prod_{k}\!\!a_k!\,, \qquad
\a^{\a}=\prod_{k=1}^{\ell} a_k^{a_k}.
\end{gather*}

Let
\beq\label{prod_e0}
\cV_{n;\a}^{(d)}:=\bigoplus_{k}R^0\pi_*\big(\sL^{\otimes a_k}\big)
 \lra \overline{Q}_{0,2}(\Gr,d),\eeq
\beq\label{prod_e}
\dot\cV_{n;\a}^{(d)}:=\bigoplus_{k}R^0\pi_*\big(\sL^{\otimes a_k}(-\si_1)\big)
 \lra \overline{Q}_{0,2}(\Gr,d)\eeq
 and
 \beq\ddot\cV_{n;\a}^{(d)}:=\bigoplus_{k}R^0\pi_*\big(\sL^{\otimes a_k}(-\si_2)\big)
 \lra \overline{Q}_{0,2}(\Gr,d).
 \eeq
In particular, $\cV_{n;\a}^{(d)}$, $\dot\cV_{n;\a}^{(d)}$ and $\ddot\cV_{n;\a}^{(d)}$ are locally free sheaves.
The $J$-functions associated to stable quotients are defined as
\beq\label{cZ1}
\dot{Z}_{n;\a}(x,\hbar,q) :=
1+\sum_{d=1}^{\infty}q^d \ev_{1*}\!\!\left[\frac{\mathbf{e}(\dot\cV_{n;\a}^{(d)})}{\hbar\!-\!\psi_1}\right]
\in H^*(\Gr)[\hbar^{-1}]\big[\big[q\big]\big],\eeq

\beq\label{ZZdfn_e1}
\ddot{Z}_{n;\a}(x,\hbar,q) :=
1+\sum_{d=1}^{\infty}q^d \ev_{1*}\!\!\left[\frac{\mathbf{e}(\ddot\cV_{n;\a}^{(d)})}{\hbar\!-\!\psi_1}\right]
\in H^*(\Gr)[\hbar^{-1}]\big[\big[q\big]\big].\eeq

 The double $J$-function is defined as
\beq\label{ZZdfn1}
\dot{Z}_{n;\a}(\hbar_1,\hbar_2,q) :=
\frac{[\Delta]}{\hbar_1+\hbar_2}+\sum_{d=1}^{\infty}q^d\big\{\ev_1\!\times\!\ev_2\}_*\!\!\left[\frac{\mathbf{e}(\dot\cV_{n;\a}^{(d)})}
{(\hbar_1\!-\!\psi_1)(\hbar_2\!-\!\psi_2)}\right]
\eeq
$$
\in H^*(\Gr)[\hbar_1^{-1},\hbar_2^{-1}]\big[\big[q\big]\big].$$

Let
\beq\label{Zgadfn_e1}\begin{split}
\dot{Z}_{\gamma^k_j}(\hbar,q) &:= \gamma^k_j(x)+
\sum_{d=1}^{\infty}\!q^d \ev_{1*}\!\!\left[\frac{\mathbf{e}(\dot\cV_{n;\a}^{(d)})\ev_2^*\gamma^k_j(x)}
{\hbar\!-\!\psi_1}\right]\in H^*(\Gr)[\hbar^{-1}]\big[\big[q\big]\big],\\
\ddot{Z}_{\gamma^k_j}(\hbar,q) &:= \gamma^k_j(x)+
\sum_{d=1}^{\infty}\!q^d \ev_{1*}\!\!\left[\frac{\mathbf{e}(\ddot\cV_{n;\a}^{(d)})\ev_2^*\gamma^k_j(x)}
{\hbar\!-\!\psi_1}\right]\in H^*(\Gr)[\hbar^{-1}]\big[\big[q\big]\big].
\end{split}\eeq

Let
\begin{eqnarray*}
\dot{Y}_{n;\a}(x,\hbar,q)=\sum_{d\ge0}q^d \dot{Y}_d(\hbar),\\
\ddot{Y}_{n;\a}(x,\hbar,q)=\sum_{d\ge0}q^d \ddot{Y}_d(\hbar),
\end{eqnarray*}
where
\begin{eqnarray*}
\dot{Y}_d(\hbar)=(-1)^d\sum_{d_1+d_2=d}\frac{(x_1-x_2+(d_1-d_2)\hbar)\prod_{k=1}^{\ell}\prod_{l=1}^{a_kd}(a_k(x_1+x_2)+l\hbar)}{(x_1-x_2)\prod_{i=1}^2\prod_{l=1}^{d_i}\bigg((x_i+l\hbar)^n-x_i^n\bigg)},\\
\ddot{Y}_d(\hbar)=(-1)^d\sum_{d_1+d_2=d}\frac{(x_1-x_2+(d_1-d_2)\hbar)\prod_{k=1}^{\ell}\prod_{l=0}^{a_kd-1}(a_k(x_1+x_2)+l\hbar)}{(x_1-x_2)\prod_{i=1}^2\prod_{l=1}^{d_i}\bigg((x_i+l\hbar)^n-x_i^n\bigg)}.
\end{eqnarray*}

Denote by 

\[ \dot{I}_{n;\a}(q):=
  \begin{cases}
    1      & \quad \text{if } |\a|<n\\
    \dot{Y}_{n;\a}(0,0,1,q)  & \quad \text{if } |\a|=n\\
  \end{cases}
,\]

 $\ddot{I}_{n;\a}(q):=1$. Let $D^{k,j}$, $\dot{C}^{(r)}_{k,i;s,j}, \ddot{C}^{(r)}_{k,i;s,j}\!\in\!\QQ[[q]]$  be defined as in Section 4.
\begin{theo}\label{thm1}
If $\ell\!\in\!\ZZ^{\ge0}$, $n\!\in\!\ZZ^+$, and $\a\!\in\!(\ZZ^{>0})^{\ell}$ are such that
$|\a|\!\le\!n$, then
$$ \dot{Z}_{n;\a}(x,\hbar,q)=\frac{\dot{Y}_{n;\a}(x,\hbar,q)}{\dot{I}_{n;\a}(q)}
\in H^*(\Gr)[\hbar^{-1}]\big[\big[q\big]\big],$$
and
$$ \ddot{Z}_{n;\a}(x,\hbar,q)=\frac{\ddot{Y}_{n;\a}(x,\hbar,q)}{\ddot{I}_{n;\a}(q)}
\in H^*(\Gr)[\hbar^{-1}]\big[\big[q\big]\big].$$
\end{theo}

\begin{theo}\label{thm2}
If $\ell\!\in\!\ZZ^{\ge0}$, $n\!\in\!\ZZ^+$, and $\a\!\in\!(\ZZ^{>0})^{\ell}$ are such that
$|\a|\!\le\!n$, then
\begin{equation*}
\dot{Z}_{\gamma^k_j}(\hbar,q)=\dot{Y}_{\gamma^k_j}(\hbar,q)\!\in\!H^*(\Gr)[\hbar^{-1}][[q]],
\end{equation*}
and
\begin{equation*}
\ddot{Z}_{\gamma^k_j}(\hbar,q)=\ddot{Y}_{\gamma^k_j}(\hbar,q)\!\in\!H^*(\Gr)[\hbar^{-1}][[q]],
\end{equation*}
where
\begin{equation*}
\dot{Y}_{\gamma^k_j}(\hbar,q):= D^{k,j}\dot{Y}_{n;\a}(x,\hbar,q)+\sum_{t=1}^{k}\sum_{s=0}^{k-t}\dot{C}^{(t)}_{k,j;s,i}(q)\hbar^{k-t-s}
D^{s,i}\dot{Y}_{n;\a}(x,\hbar,q)\!\in\!\QQ(x,\hbar)[[q]],
\end{equation*}
\begin{equation*}
\ddot{Y}_{\gamma^k_j}(\hbar,q):= D^{k,j}\ddot{Y}_{n;\a}(x,\hbar,q)+\sum_{t=1}^{k}\sum_{s=0}^{k-t}\ddot{C}^{(t)}_{k,j;s,i}(q)\hbar^{k-t-s}
D^{s,i}\ddot{Y}_{n;\a}(x,\hbar,q)\!\in\!\QQ(x,\hbar)[[q]].
\end{equation*}
\end{theo}

\begin{theo}\label{main1'}If $\ell\!\in\!\ZZ^{\ge0}$, $n\!\in\!\ZZ^+$, and $\a\!\in\!(\ZZ^{>0})^{\ell}$ are such that
$|\a|\!\le\!n$, then
\beq
\dot{Z}_{n;\a}(\hbar_1,\hbar_2,q) =\frac{1}{\hbar_1+\hbar_2}\sum_{k,j}\dot{Y}_{\gamma^k_j}(\hbar_1,q)\ddot{Y}_{\gamma^{2(n-2)-k}_j}(\hbar_2,q).
\eeq
\end{theo}
\vsp
\noindent{\bf Acknowledgment}:   The author thanks  Wanmin Liu for helpful discussions. This work was supported by Start-up Fund of Hunan University.

\section{Equivariant cohomology}
Let $\Gr:=\Gr(2,n)$ be the space of all $2$-dimensional subspaces of $\CC^n$. It has a torus action induced by the standard torus $\TT=(\CC^*)^n$ action on $\CC^n$. The $\TT$ fixed points of $\Gr$ are the 2-planes spanned by $\langle e_i,e_j\rangle$, where $\{e_i\}_{i=1}^n$ is the standard basis of $\CC^n$. We label $\langle e_i,e_j\rangle$ by $p_{ij}$. Let $\mathbf{n}^2$ be the set of all the pairs $ij$ with $1\le i\neq j\le n$. Thus $\{p_{ij}\}$ with $ij\in \mathbf{n}^2$ is the collection of all the fixed points.

Denote by $H^*(B\TT,\QQ):=\QQ[\alpha_1,\cdots,\alpha_n]$, where $\alpha_i\!=\!\pi_i^*c_1(\gamma^*)$. Here
$$\pi_i\!: (\PP^{\infty})^n\lra\PP^{\infty} \qquad\hbox{and}\qquad
\gamma\lra\PP^{\infty}$$
are the projection onto the $i$-th component and the tautological line bundle respectively. Denote by $\alpha=(\alpha_1,\cdots,\alpha_n)$. Let $\QQ_{\alpha}$ be the fractional field of $\QQ[\alpha_1,\cdots,\alpha_n]$.

Let $\mathbf{x}_1,\mathbf{x}_2$ be the equivariant Chern root of the dual universal bundle, who's non-equivariant counterpart are $x_1,x_2$.  Let $\mathbf{x}:=(\mathbf{x}_1,\mathbf{x}_2)$, then
$$
H^*_{\TT}(\Gr)\cong \QQ[\alpha_1,\cdots,\alpha_n][e_1(\mathbf{x}|\alpha),e_2(\mathbf{x}|\alpha)]/<h_{n-1}(\mathbf{x}|\alpha),h_{n}(\mathbf{x}|\alpha)>,
$$
where $e_i(\mathbf{x}|\alpha)$ and $h_{i}(\mathbf{x}|\alpha)$ are factorial elementary and complete symmetric functions which symmetric in $\mathbf{x}$, see \cite[Proposition 5.2]{Mi}.

Let $$
\phi_{ij}=\prod_{k\neq i,j}(\mathbf{x}_1-\alpha_k)\prod_{l\neq i,j}(\mathbf{x}_2-\alpha_l).
$$
Then $\phi_{ij}$ is the equivariant Poincar\'{e} dual of the points $p_{ij}\in H_{\TT}^*(\Gr)$. The Artiyah-Bott localization theorem states that
\beq
\int_{\Gr}\eta=\frac{1}{2}\sum_{ij\in\mathbf{n}^2}\int_{p_{ij}}\frac{\eta|_{p_{ij}}}{\mathbf{e}(N_{p_{ij}/\Gr})},\quad\text{for all }\eta\in H^*_{\TT}(\Gr).
\eeq
Therefore
\beq\label{AB}
\eta|_{p_{ij}}=\int_{\Gr}\eta \phi_{ij}.
\eeq

So for $\eta\in H^*_{\TT}(\Gr)$
\beq\label{zero}
\eta|_{p_{ij}}=0 \text{ for all } ij\in\mathbf{n}^2\quad\Longleftrightarrow\quad \eta=0.
\eeq

Let $\cE\!\lra\!\Gr$ be the universal bundle.
Thus, the action of $\TT$ on $\Gr$ naturally lifts to an action
on~$\cE$ and the Euler class
\begin{equation}\label{garestr_e}
\mathbf{e}\big(\det\cE^{\vee}\big)\big|_{p_{ij}}=\alpha_i+\alpha_j, \qquad\forall~i,j=1,2,\ldots,n.
\end{equation}
The $\TT$-action on $\Gr$ also has a natural lift to the tangent bundle
$T_{\Gr}\!\lra\!\Gr$ so that  there is a short exact sequence
$$0\lra \cE^{\vee}\otimes\cE
\lra \cE^{\vee}\otimes\CC^n
\lra T_{\Gr} \lra0.$$
Therefore
\beq
\mathbf{e}(T_{\Gr})|_{p_{ij}}=\frac{\prod_{k\neq i}(\alpha_i-\alpha_k)\prod_{l\neq j}(\alpha_j-\alpha_l)}{(\alpha_i-\alpha_j)(\alpha_j-\alpha_i)},\quad \forall i,j.
\eeq

\section{Algebraic calculation}

In this section we describe a number of properties of hypergeometric series which determine them uniquely.

If $R$ is a ring, denote by
$$R\Lau{\hb}:= R[[\hb^{-1}]]+R[\hb]$$
the $R$-algebra of Laurent series in $\hb^{-1}$ (with finite principal part).
For $f\!\in\!R[[q]]$ and $d\!\in\!\ZZ^{\ge0}$,
let $\coeff{f}_{q;d}\!\in\!R$ be the coefficient of~$q^d$ in~$f$.
If
$$\cF(\hb,q)=\sum_{d=0}^{\infty}
\bigg(\sum_{r=-N_d}^{\infty}\!\!\!\cF_d^{(r)}\hb^{-r}\bigg)q^d
\in R\Lau{\hb}\big[\big[q\big]\big]$$
for some $\cF_d^{(r)}\!\in\!R$, we define
$$\cF(\hb,q) \cong \sum_{d=0}^{\infty}
\bigg(\sum_{r=-N_d}^{p-1}\!\!\!\cF_d^{(r)}\hb^{-r}\bigg)q^d \quad (\mathrm{mod}\ \hb^{-p}),$$
i.e.~we drop $\hb^{-p}$ and higher powers of $\hb^{-1}$,
instead of higher powers of~$\hb$.
If $R$ is a field, let
$$R(\hb) \lhook\joinrel\lra R\Lau{\hb}$$
be the embedding given by taking the Laurent series of rational functions at $\hb^{-1}\!=\!0$.

If $f\!=\!f(z)$ is a rational function in $z$ and possibly with some other
variables, then for any $z_0\!\in\!\PP^1\!\supset\!\CC$, denote the residue of the 1-form~$f\text{d} z$ at $z_0$ by
\beq\mathrm{Res}_{z=z_0}f(z) := \frac{1}{2\pi \sqrt{-1}}\oint f(z)\mathrm{d} z,\eeq
where the integral is taken over a positively oriented loop around $z\!=\!z_0$
with no other singular points of $f\mathrm{d} z$.
For any collection of points $z_1,\ldots,z_k\!\in\!\PP^1$, let
\beq\mathrm{Res}_{z=z_1,\ldots,z_k}f(z)
:=\sum_{i=1}^{k}\mathrm{Res}_{z=z_i}f(z).\eeq
By the Residue Theorem on $S^2$,
$$\sum_{z_0\in S^2}\mathrm{Res}_{z=z_0}f(z)=0$$
for every  rational function $f\!=\!f(z)$ on $S^2\!\supset\!\CC$.
If $f$ is regular at $z\!=\!0$, let $\left\llbracket f\right\rrbracket_{z;p}$ be
the coefficient of $z^p$ in the power series expansion of $f$ around $z\!=\!0$.

\begin{defi}\label{recur_dfn}
Let $C:=(C_{ij}^{ik}(d))_{d\in\ZZ^{+}}\cup (C_{ij}^{kj}(d))_{d\in\ZZ^{+}}$ be any collection of elements of~$\QQ_{\alpha}$, with $ij,ik,kj\in \mathbf{n}^2$.
A~symmetric power series $\cF\!\in\!\QQ_{\alpha}[\mathbf{x}_1,\mathbf{x}_2]\Lau{\hb}[[q]]$ is \emph{$C$-recursive} if
the following holds:
if $d^*\!\in\!\ZZ^{\ge0}$ is such~that
$$\bigcoeff{\cF(\mathbf{x}_1\!=\!\alpha_i, \mathbf{x}_2\!=\!\alpha_j,\hb,q)}_{q;d^*-d}\in \QQ_{\alpha}(\hb)
\subset \QQ_{\alpha}\Lau{\hb},
\qquad\forall\,d\le d^*,\,1\le i\neq j\le n,$$
and $\bigcoeff{\cF(\alpha_i,\alpha_j,\hb,q)}_{q;d}$ is regular at
$\hb\!=\!(\alpha_k\!-\!\alpha_i)/d$ and $\hb\!=\!(\alpha_k\!-\!\alpha_j)/d$ for all $d\!<\!d^*$, $k\!\neq\!i$ and $k\!\neq\!j$, then
\begin{eqnarray*}
\bigcoeff{\cF(\alpha_i,\alpha_j,\hb,q)}_{q;d^*}-
\sum_{d=1}^{d^*}\sum_{k\neq j}\frac{C_{ij}^{ik}(d)}{\hb-\frac{\alpha_k-\alpha_j}{d}}
\bigcoeff{\cF(\alpha_i,\alpha_k,z,q)}_{q;d^*-d}\big|_{z=\frac{\alpha_k-\alpha_j}{d}}\\-\sum_{d=1}^{d^*}\sum_{k\neq i}\frac{C_{ij}^{kj}(d)}{\hb-\frac{\alpha_k-\alpha_i}{d}}
\bigcoeff{\cF(\alpha_k,\alpha_j,z,q)}_{q;d^*-d}\big|_{z=\frac{\alpha_k-\alpha_i}{d}}
\in \QQ_{\alpha}\big[\hb,\hb^{-1}\big]\subset \QQ_{\alpha}\Lau{\hb}.\end{eqnarray*}
\end{defi}

Thus, if $\cF\!\in\!\QQ_{\alpha}[\mathbf{x}_1,\mathbf{x}_2]\Lau{\hb}[[q]]$ is $C$-recursive, for any collection~$C$, then
$$\cF(\mathbf{x}_1\!=\!\alpha_i,\mathbf{x}_2\!=\!\alpha_j,\hb,q)\in \QQ_{\alpha}(\hb)\big[\big[q\big]\big]
\subset \QQ_{\alpha}\Lau{\hb}[[q]],
\qquad\forall\,ij\in \mathbf{n}^2,$$
as can be seen by induction on $d$. Moreover,
\begin{eqnarray}
\label{recurdfn_e2}
\cF(\alpha_i,\alpha_j,\hb,q)&=&\sum_{d=0}^{\infty}\sum_{r=-N_d}^{N_d}
\cF_{ij}^r(d)\hb^rq^d+
\sum_{d=1}^{\infty}\sum_{k\neq j}\frac{C_{ij}^{ik}(d)q^d}{\hb-\frac{\alpha_k-\alpha_j}{d}}
\cF(\alpha_i,\alpha_k,(\alpha_k\!-\!\alpha_j)/d,q)\\
&&+
\sum_{d=1}^{\infty}\sum_{k\neq i}\frac{C_{ij}^{kj}(d)q^d}{\hb-\frac{\alpha_k-\alpha_i}{d}}
\cF(\alpha_k,\alpha_j,(\alpha_k\!-\!\alpha_i)/d,q), \qquad\forall~ij,ik,kj\!\in\!\mathbf{n}^2,\nonumber
\end{eqnarray}
for some $\cF_{ij}^r(d)\!\in\!\QQ_{\alpha}$.
The nominal issue with defining $C$-recursivity by~(\ref{recurdfn_e2}), as is normally done,
is that a priori the evaluation
of $\cF(\alpha_i,\alpha_j,\hb,q)$ at $\hb\!=\!(\alpha_k\!-\!\alpha_i)/d$ and $\hb\!=\!(\alpha_k\!-\!\alpha_j)/d$ need not be well-defined,
since $\cF(\alpha_i, \alpha_j,\hb,q)$ is a power series in~$q$ with coefficients in
the Laurent series in~$\hb^{-1}$;
a priori they  may not converge anywhere.
However, taking the coefficient of each power of~$q$ in (\ref{recurdfn_e2})
shows by induction on the degree~$d$
that this evaluation does make sense. This is the substance of Definition~\ref{recur_dfn}.

\begin{defi}\label{SPC_dfn}
Let $\eta\!\in\!\QQ[\mathbf{x}]$ be a homogeneous polynomial such that taking value at $\mathbf{x}\!=\!(\alpha_i,\alpha_j)$, $\eta(\alpha_i,\alpha_j)\!\in\!\QQ_{\alpha}$ 
is well-defined and nonzero for every $1\le i,j\le n$.
For any $\cF,\cF'\in \QQ_{\alpha}[\mathbf{x}_1,\mathbf{x}_2]\Lau{\hb}[[q]]$, and symmetric on $\mathbf{x}_1,\mathbf{x}_2$, define
\beq\label{PhiZdfn_e}
\Phi^\eta_{\cF,\cF'}(\hb,z,q):= \frac{1}{2}\sum_{1\le i\neq j\le n}
\frac{\eta(\alpha_i,\alpha_j)\e^{(\alpha_i+\alpha_j)z}}{\prod\limits_{k\neq i,j}(\alpha_i\!-\!\alpha_k)\prod\limits_{k\neq i,j}(\alpha_j\!-\!\alpha_k)}
\cF\big(\alpha_i,\alpha_j,\hb,q \e^{\hb z}\big)\cF'(\alpha_i,\alpha_j,-\hb,q),
\eeq
which is in $\QQ_{\alpha}\Lau{\hb}[[z,q]]$. The pair $(\cF,\cF')$
satisfies $\eta$-mutual polynomiality condition ($\eta$-MPC) if
$\Phi^\eta_{\cF,\cF'}\in \QQ_{\alpha}[\hb][[z,q]]$. If $\cF=\cF'$, we call that $\cF$ satisfies $\eta$-self polynomiality condition ($\eta$-SPC).
\end{defi}

\begin{prop}[{\cite[Proposition 6.3]{Zinger1}}]\label{uniqueness_prp}
Let $\cF,\cF'\in \QQ_{\alpha}[\mathbf{x}_1,\mathbf{x}_2]\Lau{\hb}[[q]]$.
If $\cF$ and $\cF'$ are $C$-recursive, for some collection
$C=(C_{ij}^{ik}(d))_{d\in\ZZ^{+};ij,ik\in\mathbf{n}^2}$ of elements of~$\QQ_{\alpha}$, and $(\cF,\cF')$
satisfies $\eta$-MPC, with
$$\coeff{\cF(\mathbf{x}_1\!=\!\alpha_i,\mathbf{x}_2\!=\!\alpha_j,\hb,q)}_{q,0}\neq0,$$
for all $ij\!\in\!\mathbf{n}^2$, then  $\cF'\cong0~(\mathrm{mod}\ \hb^{-1})$ if and only if $\cF'=0$.
\end{prop}

\begin{prop}\label{Mut}
Let
$$
\cY=\sum_{d=0}^{\infty}q^d\frac{h_d(\mathbf{x}_1,\mathbf{x}_2,\hbar)}{g_d(\mathbf{x}_1,\mathbf{x}_2,\hbar)},
$$
and
$$
\cZ=\sum_{d=0}^{\infty}q^d\frac{\tilde{h}_d(\mathbf{x}_1,\mathbf{x}_2,\hbar)}{\tilde{g}_d(\mathbf{x}_1,\mathbf{x}_2,\hbar)},
$$
where $h_d,g_d$,$\tilde{h}_d,\tilde{g}_d$ are polynomials in $\QQ_{\alpha}[\mathbf{x}_1,\mathbf{x}_2,\hbar]$ symmetric on $\mathbf{x}_1,\mathbf{x}_2$, and $g_d$, $\tilde{g}_d$ are the product of $\bigg(\prod_{i=1}^n(\mathbf{x}_1-\alpha_i+l\hbar)-\prod_{i=1}^n(\mathbf{x}_1-\alpha_i)\bigg)$ and $\bigg(\prod_{i=1}^n(\mathbf{x}_2-\alpha_i+l\hbar)-\prod_{i=1}^n(\mathbf{x}_2-\alpha_i)\bigg)$ with $l\in\ZZ^{>0}$, $1\le i\le n$.
Let $\eta$ be defined as in Definition \ref{SPC_dfn}, then $$\Phi^{\eta}_{\cY,\cZ}\in \QQ_{\alpha}[\hb][[z,q]].$$
\end{prop}
\begin{proof} Denote by
$$\widetilde{\Phi}:=\frac{\eta(x)\e^{(\mathbf{x}_1+\mathbf{x}_2)z}(\mathbf{x}_1-\mathbf{x}_2)(\mathbf{x}_2-\mathbf{x}_1)}{\prod\limits_{k}(\mathbf{x}_1\!-\!\alpha_k)\prod\limits_{k}(\mathbf{x}_2\!-\!\alpha_k)}
\cY\big(\mathbf{x}_1,\mathbf{x}_2,\hb,q \e^{\hb z}\big)\cZ(\mathbf{x}_1,\mathbf{x}_2,-\hb,q).
$$
Since $g_d$ and $\tilde{g}_d$ are not zero when $(\mathbf{x}_1,\mathbf{x}_2)=(\alpha_i,\alpha_j)$ for all $i,j$,
$$\Phi^{\eta}_{\cY,\cZ}=\frac{1}{2}\sum_{1\le i\neq j\le n}\mathrm{Res}_{\mathbf{x}_1=\alpha_i}\mathrm{Res}_{\mathbf{x}_2=\alpha_j}(\widetilde{\Phi}).$$

Thus, by the Residue Theorem on~$S^2$,
$$\Phi^{\eta}_{\cY,\cZ} =\frac{1}{2}\mathrm{Res}_{\mathbf{x}_1=\alpha_i}\mathrm{Res}_{\mathbf{x}_2=\alpha_j}(\widetilde{\Phi})=
-\frac{1}{2}\mathrm{Res}_{\mathbf{x}_1=0,\infty}\mathrm{Res}_{\mathbf{x}_2=0,\infty}(\widetilde{\Phi}).
$$
Since $g_d$ and $\tilde{g}_d$ are not zero when $(\mathbf{x}_1,\mathbf{x}_2)=(0,0)$,
$$
\mathrm{Res}_{\mathbf{x}_2=0}(\widetilde{\Phi})=0.
$$
 Therefore
$$
\mathrm{Res}_{\mathbf{x}_1=0,\infty}\mathrm{Res}_{\mathbf{x}_2=0}(\widetilde{\Phi})=0.
$$

On the other hand
\begin{eqnarray*}
&&\mathrm{Res}_{\mathbf{x}_2=\infty}(\widetilde{\Phi})\\
&=&\frac{ \e^{\mathbf{x}_1z}}{\prod\limits_{k}(\mathbf{x}_1\!-\!\alpha_k)}\sum_{d_1,d_2=0}^{\infty}\sum_{p=0}^{\infty}
\frac{z^{n-1+p+(\deg_{\mathbf{x}2} g_{d_1}+\deg_{\mathbf{x}_2}\ti g_{d_2}-\deg_{\mathbf{x}_2} h_{d_1}-\deg_{\mathbf{x}_2}\ti h_{d_2}-\deg \eta)}}
{(n-1+p+(\deg_{\mathbf{x}_2} g_{d_1}+\deg_{\mathbf{x}_2}\ti g_{d_2}-\deg_{\mathbf{x}_2} h_{d_1}-\deg_{\mathbf{x}_2}\ti h_{d_2})-\deg \eta)!}\\
&&\quad \times q^{d_1+d_2}\e^{\hbar (d_1+d_2)z}
\times\left\llbracket
\frac{\eta(\omega_2,\mathbf{x}_1)(\omega_2\mathbf{x}_1-1)^2h_d(\mathbf{x}_1,\omega_2,\hbar)\tilde{h}_d(\mathbf{x}_1,\omega_2,\hbar)}{\prod\limits_{k}(1\!-\!\alpha_k\omega_2)B_1(\mathbf{x}_1,\hbar)}\right\rrbracket_{\omega_2,p},
\end{eqnarray*}
where $B_1(\mathbf{x}_1,\hbar)$ is a product of $\bigg(\prod_{i=1}^n(\mathbf{x}_1-\alpha_i+l\hbar)-\prod_{i=1}^n(\mathbf{x}_1-\alpha_i)\bigg)$, which is not zero when $\mathbf{x}_1=0$. Thus $\mathrm{Res}_{\mathbf{x}_1=0}\mathrm{Res}_{\mathbf{x}_2=\infty}(\widetilde{\Phi})=0$,  and the $(d_1,d_2,p)$-summand of $\mathrm{Res}_{\mathbf{x}_1=\infty}\mathrm{Res}_{\mathbf{x}_2=\infty}(\widetilde{\Phi})$
 is $q^{d_1+d_2}$ times an element of $\QQ_{\alpha}[\hbar]\left\llbracket z\right\rrbracket$.
Therefore
$$\mathrm{Res}_{\mathbf{x}_1=0,\infty}\mathrm{Res}_{\mathbf{x}_2=\infty}(\widetilde{\Phi})\in \QQ_{\alpha}[\hbar]\left\llbracket q,z\right\rrbracket.
$$
\end{proof}

\subsection{Recursive}
\begin{defi}\label{recur_dfn_mul}
Let $\fC:=(\fC_{ij}^{ik}(d))_{d\in\ZZ^{+}}\cup (\fC_{ij}^{kj}(d))_{d\in\ZZ^{+}}$ be any collection of elements of~$\QQ_{\alpha}$, with $ij,ik,kj\in \bar{\mathbf{n}}^2$, where $\bar{\mathbf{n}}^2$ is the collection of the pairs $ij$ with $1\le i,j\le n$.
A power series $\cF\!\in\!\QQ_{\alpha}[\mathbf{x}_1,\mathbf{x}_2]\Lau{\hb}[[\mathbf{q}]]$ is $\fC$-recursive if
the following holds:
\begin{eqnarray}
\label{recurdfn} 
\cF(\alpha_i,\alpha_j,\hb,\mathbf{q})&=&\sum_{\mathbf{d}\in(\ZZ^{\ge0})^2}^{\infty}\sum_{r=-N_{\mathbf{d}}}^{N_{\mathbf{d}}}
\cF_{ij}^r(\mathbf{d})\hb^r\mathbf{q}^{\mathbf{d}}+
\sum_{d=1}^{\infty}\sum_{k\neq j}\frac{\fC_{ij}^{ik}(d)q_2^d}{\hb-\frac{\alpha_k-\alpha_j}{d}}
\cF(\alpha_i,\alpha_k,(\alpha_k\!-\!\alpha_j)/d,\mathbf{q})\\
&&+
\sum_{d=1}^{\infty}\sum_{k\neq i}\frac{\fC_{ij}^{kj}(d)q_1^d}{\hb-\frac{\alpha_k-\alpha_i}{d}}
\cF(\alpha_k,\alpha_j,(\alpha_k\!-\!\alpha_i)/d,\mathbf{q}) \qquad\forall\, ik,kj\!\in\!\bar{\mathbf{n}}^2,\nonumber\end{eqnarray}
with $i\neq j$.
\end{defi}
For $\cF\in \QQ_{\alpha}[\mathbf{x}_1,\mathbf{x}_2]\Lau{\hb}[[\mathbf{q}]]$ which is symmetry on $\mathbf{x}$, Let
\beq\label{newdef}
\overline{\cF}=\cF|_{q_1=q_2=qe^{\pi i}}+\frac{\hbar(q_1\frac{d}{dq_1}-q_2\frac{d}{dq_2})}{\mathbf{x}_1-\mathbf{x}_2}\cF|_{q_1=q_2=qe^{\pi i}}\in\QQ_{\alpha}[\mathbf{x}_1,\mathbf{x}_2]\Lau{\hb}[[q]].
\eeq

\begin{prop}\label{cFrec_lmm2}
If $\cF$ is a symmetry function on $\mathbf{x}$ and $\fC$-recursive as in Definition \ref{recur_dfn_mul}, then $\overline{\cF}$ is $C$-recursive as in Definition \ref{recur_dfn} with
\begin{equation}\label{C-recursive}
C_{ij}^{ik}(d)=(-1)^d\frac{\alpha_i-\alpha_k}{\alpha_i-\alpha_j}\fC_{ij}^{ik}(d),  \text{ and }
C_{ij}^{kj}(d)=(-1)^d\frac{\alpha_k-\alpha_j}{\alpha_i-\alpha_j}\fC_{ij}^{kj}(d).
\end{equation} 
\end{prop}
\begin{proof}
$\cF$ is $\fC$-recursive, thus
$$
\frac{\fC_{ij}^{ik}(d)q_2^d}{\hb-\frac{\alpha_k-\alpha_j}{d}}
\cF(\alpha_i,\alpha_k,(\alpha_k\!-\!\alpha_j)/d,\mathbf{q})=\mathrm{Res}_{z=\frac{\alpha_k-\alpha_j}{d}}\bigg\{\frac{1}{\hb\!-\!z}\cF(\alpha_i,\alpha_j,z,\mathbf{q})\bigg\},
$$
$$
\mathrm{Res}_{z=0,\infty}\bigg\{\frac{1}{\hb\!-\!z}\cF(\alpha_i,\alpha_j,z,\mathbf{q})\bigg\}\in \QQ_{\alpha}[\hbar,\hbar^{-1}][[q]].
$$
Therefore
\begin{eqnarray*}
&&\mathrm{Res}_{z=0,\infty}\bigg\{\frac{1}{\hb\!-\!z}\overline{\cF}(\alpha_i,\alpha_j,z,\mathbf{q})\bigg\}\\
&=&\mathrm{Res}_{z=0,\infty}\bigg\{\frac{1}{\hb\!-\!z}\cF|_{q_1=q_2=qe^{\pi i}}-\frac{z}{z-\hbar}\frac{(q_1\frac{d}{dq_1}-q_2\frac{d}{dq_2})}{\alpha_i-\alpha_j}\cF|_{q_1=q_2=qe^{\pi i}}\bigg\}\\
&=&\mathrm{Res}_{z=0,\infty}\bigg\{\frac{1}{\hb\!-\!z}\cF(\alpha_i,\alpha_j,z,\mathbf{q})|_{q_1=q_2=qe^{\pi i}}\bigg\}\\
&&+\mathrm{Res}_{z=0,\infty}\bigg\{(-1+\frac{\hbar}{\hb\!-\!z})\frac{(q_1\frac{d}{dq_1}-q_2\frac{d}{dq_2})}{\alpha_i-\alpha_j}\cF(\alpha_i,\alpha_j,z,\mathbf{q})|_{q_1=q_2=qe^{\pi i}}\bigg\}.
\end{eqnarray*}
So 
$$
\mathrm{Res}_{z=0,\infty}\bigg\{\frac{1}{\hb\!-\!z}\overline{\cF}(\alpha_i,\alpha_j,z,\mathbf{q})\bigg\}\in \QQ_{\alpha}[\hbar,\hbar^{-1}][[q]].
$$

\begin{eqnarray*}
&&\mathrm{Res}_{z=\frac{\alpha_k-\alpha_j}{d}}\bigg\{\frac{1}{\hb\!-\!z}\overline{\cF}(\alpha_i,\alpha_j,z,q)\bigg\}\\
&=&\mathrm{Res}_{z=\frac{\alpha_k-\alpha_j}{d}}\bigg\{\frac{1}{\hb\!-\!z}\cF|_{q_1=q_2=qe^{\pi i}}+\frac{z(q_1\frac{d}{dq_1}-q_2\frac{d}{dq_2})}{\alpha_i-\alpha_j}\cF|_{q_1=q_2=qe^{\pi i}}\bigg\}\\
&=&\frac{\fC_{ij}^{ik}(d)q_2^d}{\hb-\frac{\alpha_k-\alpha_j}{d}}
\cF(\alpha_i,\alpha_k,(\alpha_k\!-\!\alpha_j)/d,\mathbf{q})|_{q_1=q_2=qe^{\pi i}}\\
&&+\frac{\frac{\alpha_k-\alpha_j}{d}(q_1\frac{d}{dq_1}-q_2\frac{d}{dq_2})}{\alpha_i-\alpha_j}\frac{\fC_{ij}^{ik}(d)q_2^d}{\hb-\frac{\alpha_k-\alpha_j}{d}}
\cF(\alpha_i,\alpha_k,(\alpha_k\!-\!\alpha_j)/d,\mathbf{q})|_{q_1=q_2=qe^{\pi i}}\\
&=&\frac{\alpha_i-\alpha_k}{\alpha_i-\alpha_j}\frac{\fC_{ij}^{ik}(d)q_2^d}{\hb-\frac{\alpha_k-\alpha_j}{d}}
\cF(\alpha_i,\alpha_k,(\alpha_k\!-\!\alpha_j)/d,\mathbf{q})|_{q_1=q_2=qe^{\pi i}}\\
&&+\frac{\frac{\alpha_k-\alpha_j}{d}}{\alpha_i-\alpha_j}\frac{\fC_{ij}^{ik}(d)q_2^d}{\hb-\frac{\alpha_k-\alpha_j}{d}}
(q_1\frac{d}{dq_1}-q_2\frac{d}{dq_2})\cF(\alpha_i,\alpha_k,(\alpha_k\!-\!\alpha_j)/d,\mathbf{q})|_{q_1=q_2=qe^{\pi i}}\\
&=&(-1)^d\frac{\alpha_i-\alpha_k}{\alpha_i-\alpha_j}\frac{\fC_{ij}^{ik}(d)q^d}{\hb-\frac{\alpha_k-\alpha_j}{d}}\big\{\cF(\alpha_i,\alpha_k,(\alpha_k\!-\!\alpha_j)/d,\mathbf{q})|_{q_1=q_2=qe^{\pi i}}\\
&&+\frac{\frac{\alpha_k-\alpha_j}{d}}{\alpha_i-\alpha_k}(q_1\frac{d}{dq_1}-q_2\frac{d}{dq_2})\cF(\alpha_i,\alpha_k,(\alpha_k\!-\!\alpha_j)/d,\mathbf{q})|_{q_1=q_2=qe^{\pi i}}\big\}\\
&=&(-1)^d\frac{\alpha_i-\alpha_k}{\alpha_i-\alpha_j}\frac{\fC_{ij}^{ik}(d)q^d}{\hb-\frac{\alpha_k-\alpha_j}{d}}\overline{\cF}((\alpha_i,\alpha_k,(\alpha_k\!-\!\alpha_j)/d,q).
\end{eqnarray*}
By the Residue Theorem on $S^2$,
\begin{eqnarray*}
\overline{\cF}(\alpha_i,\alpha_j,\hbar,q)&=&\mathrm{Res}_{z=\frac{\alpha_k-\alpha_j}{d}}\bigg\{\frac{1}{\hb\!-\!z}\overline{\cF}(\alpha_i,\alpha_j,z,q)\bigg\}+\mathrm{Res}_{z=\frac{\alpha_k-\alpha_i}{d}}\bigg\{\frac{1}{\hb\!-\!z}\overline{\cF}(\alpha_i,\alpha_j,z,q)\bigg\}\\
&&+\mathrm{Res}_{z=0,\infty}\bigg\{\frac{1}{\hb\!-\!z}\overline{\cF}(\alpha_i,\alpha_j,z,q)\bigg\}.
\end{eqnarray*}
So $\overline{\cF}$ is $C$-recursive with relations as (\ref{C-recursive}).
\end{proof}

Denote by $\QQ_{\tilde{\alpha}}=\QQ(\alpha_{1;1},\cdots,\alpha_{1;n},\alpha_{2;1},\cdots,\alpha_{2;n})$. Let $\dot{\cA}, \ddot{\cA}\in\QQ_{\tilde{\alpha}}[\mathbf{x}_1,\mathbf{x}_2]\Lau{\hb}[[\mathbf{q}]]$ as
\begin{eqnarray*}
&&\dot{\cA}(\mathbf{x}_1,\mathbf{x}_2,\hbar,q)=\sum_{d\ge0} \dot{\cA}_d(\hbar),\quad \ddot{\cA}(\mathbf{x}_1,\mathbf{x}_2,\hbar,q)=\sum_{d\ge0} \ddot{\cA}_d(\hbar),\quad \text{where}\\
&&\dot{\cA}_d(\hbar)=\sum_{d_1+d_2=d}q_1^{d_1}q_2^{d_2}\frac{\prod_{r=1}^{\ell}\prod_{l=1}^{a_{r;1}d_1+a_{r;2}d_2}
(a_{r;1}\mathbf{x}_1+a_{r;2}\mathbf{x}_2+l\hbar)}{\prod_{i=1}^2\prod_{l=1}^{d_i}\bigg(\prod_{j=1}^n(\mathbf{x}_i-\alpha_{i;j}+l\hbar)-\prod_{j=1}^n(\mathbf{x}_i-\alpha_{i;j})\bigg)},\\
&&\ddot{\cA}_d(\hbar)=\sum_{d_1+d_2=d}q_1^{d_1}q_2^{d_2}\frac{\prod_{r=1}^{\ell}\prod_{l=0}^{a_{r;1}d_1+a_{r;2}d_2-1}(a_{r;1}\mathbf{x}_1+a_{r;2}\mathbf{x}_2+l\hbar)}
{\prod_{i=1}^2\prod_{l=1}^{d_i}\bigg(\prod_{j=1}^n(\mathbf{x}_i-\alpha_{i;j}+l\hbar)-\prod_{j=1}^n(\mathbf{x}_i-\alpha_{i;j})\bigg)},
\end{eqnarray*}
with $a_{k,i}\in\ZZ^{\ge0}$.
Let
\begin{eqnarray*}
\dot{\sC}_{ i_1i_2}^{k}(s;d)&=&\frac{\prod_{r=1}^{\ell}\prod_{l=1}^{a_{r;s}d}(\sum_{t=1}^2a_{r;t}\alpha_{t;i_t})+\frac{l}{d}(\alpha_{s;k}-\alpha_{s;i_s}))}
{d{\prod_{l=1}^{d}\prod_{m=1}^n\atop(l,m)\neq(d,k)}(\alpha_{s;i_s}-\alpha_{s;m}+\frac{l}{d}(\alpha_{s;k}-\alpha_{s;i_s}))},\\
\ddot{\sC}_{ i_1i_2}^{k}(s;d)&=&\frac{\prod_{r=1}^{\ell}\prod_{l=0}^{a_{r;s}d-1}(\sum_{t=1}^2a_{r;t}\alpha_{t;i_t})+\frac{l}{d}(\alpha_{s;k}-\alpha_{s;i_s}))}
{d{\prod_{l=1}^{d}\prod_{m=1}^n\atop(l,m)\neq(d,k)}(\alpha_{s;i_s}-\alpha_{s;m}+\frac{l}{d}(\alpha_{s;k}-\alpha_{s;i_s}))}.
\end{eqnarray*}
Note that $\dot{\cA}, \ddot{\cA}$ can be considered as the associated $J$-functions of $\PP^{n-1}\times\PP^{n-1}$. By \cite[Section 6, Page 477]{Zinger1}, we have the following result.
\begin{lemm}\label{cYrec_lmm}
If $\sum_{r,i}a_{r,i}\!\le \!n$,
the power series $\dot{\cA}, \ddot{\cA}$ are $\dot{\sC}_{i_1i_2}^{j}(s,d)$ and $\ddot{\sC}_{i_1i_2}^{j}(s,d)$-recursive respectively.
\end{lemm}
Therefore
\begin{eqnarray}\label{trans}
&&\dot\cA(\alpha_{1,i_1},\alpha_{2,i_2},\hb,\mathbf{q})\\
&=&\sum_{\mathbf{d}\in(\ZZ^{\ge0})^2}^{\infty}\sum_{r=-N_{\mathbf{d}}}^{N_{\mathbf{d}}}
\dot\cA_{i_1i_2}^r(\mathbf{d})\hb^r\mathbf{q}^{\mathbf{d}}+
\sum_{d=1}^{\infty}\sum_{k\neq i_2}\frac{\dot\sC_{i_1i_2}^{k}(2,d)q_2^d}{\hb-\frac{\alpha_{2,k}-\alpha_{2,i_2}}{d}}
\dot\cA(\alpha_{1,i_1},\alpha_{2,k},(\alpha_{2,k}\!-\!\alpha_{2,i_2})/d,\mathbf{q})\nonumber\\
&&+
\sum_{d=1}^{\infty}\sum_{k\neq i}\frac{\dot\sC_{i_1i_2}^{k}(1,d)q_1^d}{\hb-\frac{\alpha_{1,k}-\alpha_{1,i_1}}{d}}
\dot\cA(\alpha_k,\alpha_{2,i_2},(\alpha_{1,k}\!-\!\alpha_{1,i_1})/d,\mathbf{q}),\nonumber\end{eqnarray}

similarly for $\ddot\cA(\alpha_{1,i_1},\alpha_{2,i_2},\hb,\mathbf{q})$.

For all $i, j, r$, denote
$$
\dot{\cK}=\dot{\cA}|_{a_{r,i}=a_r,\, \alpha_{1,i}=\alpha_i,\, \alpha_{2,j}=\alpha_j},\quad
\ddot{\cK}=\ddot{\cA}|_{a_{r,i}=a_r,\, \alpha_{1,i}=\alpha_i,\, \alpha_{2,j}=\alpha_j}.
$$
Under the conditions
$$a_{r,i}=a_r,\ \alpha_{1,i}=\alpha_i,\ \alpha_{2,j}=\alpha_j, \text{ for all } i,j,r,\quad
\quad $$
$\dot{\sC}_{ i_1i_2}^{j}(s;d)$  and $\ddot{\sC}_{ i_1i_2}^{j}(s;d)$ can be written as
$$
\dot{\fC}_{ij}^{ik}(d)=\frac{\prod_{r=1}^{\ell}\prod_{l=1}^{a_{r}d}(a_{r}(\alpha_i+\alpha_j)+\frac{l}{d}(\alpha_k-\alpha_j))}
{d{\prod_{l=1}^{d}\prod_{m=1}^n\atop(l,m)\neq(d,k)}(\alpha_j-\alpha_m+\frac{l}{d}(\alpha_k-\alpha_j))},
$$

$$
\ddot{\fC}_{ij}^{ik}(d)=\frac{\prod_{r=1}^{\ell}\prod_{l=0}^{a_{r}d-1}(a_{r}(\alpha_i+\alpha_j)+\frac{l}{d}(\alpha_k-\alpha_j))}
{d{\prod_{l=1}^{d}\prod_{m=1}^n\atop(l,m)\neq(d,k)}(\alpha_j-\alpha_m+\frac{l}{d}(\alpha_k-\alpha_j))},
$$
just changing $\alpha_{s,i_1}$ to $\alpha_{i}$, $\alpha_{s,i_2}$ to $\alpha_{j}$, $\alpha_{s,k}$ to $\alpha_{k}$.

Therefore $\dot{\cK}$ and $\ddot{\cK}$ satisfy the $\fC$-recursive as defined in Definition  \ref{recur_dfn_mul}.

Denote $$\dot{\cY}_{n;\a}(\mathbf{x},\hbar,q)=\dot{\cK}|_{q_1=q_2=qe^{\pi i}}+\frac{\hbar(q_1\frac{d}{dq_1}-q_2\frac{d}{dq_2})}{\mathbf{x}_1-\mathbf{x}_2}\dot{\cK}|_{q_1=q_2=qe^{\pi i}},$$
and
$$\ddot{\cY}_{n;\a}(\mathbf{x},\hbar,q)=\ddot{\cK}|_{q_1=q_2=qe^{\pi i}}+\frac{\hbar(q_1\frac{d}{dq_1}-q_2\frac{d}{dq_2})}{\mathbf{x}_1-\mathbf{x}_2}\ddot{\cK}|_{q_1=q_2=qe^{\pi i}}.$$
Thus $\dot{Y},\ddot{Y}$ are the non-equivariant form of $\dot{\cY},\ddot{\cY}$
\beq
\dot{Y}_{n;\a}(x,\hbar,q)=\dot{\cY}_{n;\a}(\mathbf{x},\hbar,q)|_{\alpha_i=0},\quad \ddot{Y}_{n;\a}(x,\hbar,q)=\ddot{\cY}_{n;\a}(\mathbf{x},\hbar,q)|_{\alpha_i=0}.
\eeq
Let
\begin{eqnarray}\label{C}
&&\dot{C}_{ij}^{ik}(d):=(-1)^d\frac{\alpha_i-\alpha_k}{\alpha_i-\alpha_j}\frac{\prod_{r=1}^{\ell}\prod_{l=1}^{a_{r}d}(a_{r}(\alpha_i+\alpha_j)
+\frac{l}{d}(\alpha_k-\alpha_j))}{d{\prod_{l=1}^{d}\prod_{m=1}^n\atop(l,m)\neq(d,k)}(\alpha_j-\alpha_m+\frac{l}{d}(\alpha_k-\alpha_j))},
\end{eqnarray}
\begin{eqnarray}\label{C_1}
&&\ddot{C}_{ij}^{ik}(d):=(-1)^d\frac{\alpha_i-\alpha_k}{\alpha_i-\alpha_j}\frac{\prod_{r=1}^{\ell}\prod_{l=0}^{a_{r}d-1}(a_{r}(\alpha_i
+\alpha_j)+\frac{l}{d}(\alpha_k-\alpha_j))}{d{\prod_{l=1}^{d}\prod_{m=1}^n\atop(l,m)\neq(d,k)}(\alpha_j-\alpha_m+\frac{l}{d}(\alpha_k-\alpha_j))}.
\end{eqnarray}
By Proposition \ref{cFrec_lmm2} and Proposition \ref{Mut}, we have
\begin{coro}\label{cIrec_lmm2}
The power series $\dot{\cY}_{n;\a}(\mathbf{x},\hbar,q)$, $\ddot{\cY}_{n;\a}(\mathbf{x},\hbar,q)$, are $C$-recursive as in (\ref{recur_dfn}) with $\dot{C}_{ij}^{ik}(d)$ and $\ddot{C}_{ij}^{ik}(d)$ given as above. Moreover, $\dot{\cY}_{n;\a}(\mathbf{x},\hbar,q)$ satisfies $\eta$-SPC with $\eta=\langle\a\rangle(\mathbf{x}_1+\mathbf{x}_2)^{\ell}$, and $(\dot{\cY}_{n;\a}(\mathbf{x},\hbar,q),\ddot{\cY}_{n;\a}(\mathbf{x},\hbar,q))$ satisfies $1$-MPC.
\end{coro}

\subsection{Differential operators on hypergeometric functions}
Denote $\mathbf{q}=(q_1,q_2)$, and $\mathbf{x}=(\mathbf{x}_1,\mathbf{x}_2)$. Fix an element $\cY(\mathbf{x},\hbar,\mathbf{q})\!\in\!\QQ_{\tilde{\alpha}}(\mathbf{x},\hbar)[[\mathbf{q}]]$. For $\mathbf{r}\!\in\!(\ZZ^{\ge0})^2$, $\mathbf{x}^\mathbf{r}:=\mathbf{x}_1^{r_1}\mathbf{x}_2^{r_2}$. Denote
$$\left\llbracket\cY(\mathbf{x},\hbar,\mathbf{q})\right\rrbracket_{\mathbf{q};\bd}:=\frac{f_{\bd}(\mathbf{x},\hbar)}{g_{\bd}(\mathbf{x},\hbar)}$$
for some homogeneous polynomials $f_{\bd}(\mathbf{x},\hbar),g_{\bd}(\mathbf{x},\hbar)\!\in\!\QQ[\tilde{\alpha},\mathbf{x},\hbar]$
satisfying
\begin{eqnarray}\label{assumecS_e}
&&f_{0}(\mathbf{x},\hbar)=g_{0}(\mathbf{x},\hbar),\qquad \deg\,f_{\bd}-\deg\,g_{\bd}=0,\qquad\\
&&g_{\bd}\left|_{\begin{subarray}{c}\mathbf{x}=0\\\tilde{\alpha}=0\end{subarray}}\right.\!\neq\!0.
\end{eqnarray}

By \cite[Section 4]{Popa}, there exists differential operators $\fD^{\mathbf{p}}$, where $\mathbf{p}\!\in\!(\ZZ^{\ge0})^2$, such that $\fD^{\mathbf{p}}\cY(\mathbf{x},\hbar,\mathbf{q})$ in $\QQ_{\tilde{\alpha}}(\mathbf{x},\hbar)[[\mathbf{q}]]$
satisfies the following properties:
\begin{enumerate}
\item[]
There exist $\cC^{(\mathbf{r})}_{\mathbf{p},s}\!:=\!\cC^{(\mathbf{r})}_{\mathbf{p},s}(\cY)\!\in\!\QQ[\tilde{\alpha}][[\mathbf{q}]]$
with $\mathbf{p},\mathbf{r}\!\in\!(\ZZ^{\ge 0})^2$, $s\!\in\!\ZZ^{\ge0}$, such that $\left\llbracket\cC^{(\mathbf{r})}_{\mathbf{p},s}\right\rrbracket_{\mathbf{q};\bd}$ is a homogeneous polynomial in $\alpha$ of degree $\!-\!|\mathbf{r}|\!+\!s$,
\begin{equation}\label{fDpexp_0e}
\fD^{\mathbf{p}}\cY(\mathbf{x},\hbar,\mathbf{q})=\hbar^{|\mathbf{p}|}\sum\limits_{s=0}^{\infty}\sum\limits_{\mathbf{r}\in(\ZZ^{\ge 0})^k}\cC^{(\mathbf{r})}_{\mathbf{p},s}(\mathbf{q})
\mathbf{x}^{\mathbf{r}}\hbar^{-s},
\end{equation}
\begin{equation}\label{EP_Dpq;01_e}
\left\llbracket\cC^{(\mathbf{r})}_{\mathbf{p},s}\right\rrbracket_{\mathbf{q};0}\!\!\!\!=\!\delta_{\mathbf{p},\mathbf{r}}\delta_{|\mathbf{r}|,s}\quad\forall\,\mathbf{p},\mathbf{r}\!\in\!(\ZZ^{\ge0})^2,\,s\!\in\!\ZZ^{\ge 0},
\end{equation}
and
\begin{equation}\label{cCprop_e}
\cC^{(\mathbf{r})}_{\mathbf{p},|\mathbf{r}|}=
\delta_{\mathbf{p},\mathbf{r}}\quad\textnormal{if}\quad|\mathbf{r}|\!\le\!|\mathbf{p}|.
\end{equation}
\end{enumerate}

Let $\gamma^k_j(\fD)$ be a polynomial of operator $\fD$ such that

$$
\left\llbracket\gamma^k_j(\fD)\dot{\cA}|_{\tilde{\alpha}=0}\right\rrbracket_{\mathbf{q};0}=\gamma^k_j(\mathbf{x}).
$$

Denote by
\begin{eqnarray*}
&&\gamma^k_j(\fD)\dot{\cK}:=\big(\gamma^k_j(\fD)\dot{\cA}\big)|_{a_{k,i}=a_k,\,\alpha_{1,i}=\alpha_i,\,\alpha_{2,j}=\alpha_j}, \end{eqnarray*}

\begin{eqnarray*}
\overline{\fD}^{k,j}\dot{\cY}_{n;\a}(\mathbf{x},\hbar,q)&:=&\gamma^k_j(\fD)\dot{\cK}|_{q_1=q_2=qe^{\pi i}}\\
&&+\frac{\hbar(q_1\frac{d}{dq_1}-q_2\frac{d}{dq_2})}{\mathbf{x}_1-\mathbf{x}_2}\gamma^k_j(\fD)\dot{\cK}|_{q_1=q_2=qe^{\pi i}}.
\end{eqnarray*}

Define
$$
\overline{J}_{k}(\dot{\cY}_{n;\a}(\mathbf{x},\hbar,q))(\gamma^k_j(\mathbf{x})):=\left\llbracket\overline{\fD}^{k,j}\dot{\cY}_{n;\a}(\mathbf{x},\hbar,q)\right\rrbracket_{x,k}\in \QQ[[q]][\mathbf{x}_1,\mathbf{x}_2]^{S_2}_k,$$

where $\QQ[[q]][\mathbf{x}_1,\mathbf{x}_2]^{S_2}_k$ is the set of all the symmetric polynomials of $\mathbf{x}=(\mathbf{x}_1,\mathbf{x}_2)$ with degree $k$.
Then
$$ \left\llbracket\overline{J}_{k}(\dot{\cY}_{n;\a}(\mathbf{x},\hbar,q))(\gamma^k_j(\mathbf{x}))\right\rrbracket_{q,0}=\gamma^k_j(\mathbf{x}).
$$
In particular $\overline{J}_{k}(\dot{\cY})$ as an element of $\mbox{End}(\QQ[[q]][\mathbf{x}_1,\mathbf{x}_2]^{S_2}_k)$ is invertible, with $\bar{c}_{\gamma^k_i,\gamma^k_j}$ given by
$$
\overline{J}^{-1}_{k}(\dot{\cY}_{n;\a}(\mathbf{x},\hbar,q))=\sum \bar{c}_{\gamma^k_i,\gamma^k_j}\gamma^k_j(\mathbf{x}).
$$
We define
$$
\cD^{k,i}\dot{\cY}_{n;\a}(\mathbf{x},\hbar,q)=\sum\bar{c}_{\gamma^k_i,\gamma^k_j}\overline{\fD}^{k,j}\dot{\cY}_{n;\a}(\mathbf{x},\hbar,q).
$$
Then
\begin{equation}\label{fDpexp_e}
\cD^{k,i}\dot{\cY}_{n;\a}(\mathbf{x},\hbar,q)=\hbar^{k}\sum\limits_{s=0}^{\infty}\sum\limits_{r,j}\cC^{(r,j)}_{k,i,s}(q)
\gamma^r_j(\mathbf{x})\hbar^{-s},
\end{equation}
and
\begin{equation}\label{EP_Dpq;0_e}
\left\llbracket\cC^{(r,j)}_{k,i,s}\right\rrbracket_{q;0}\!\!\!\!=\!\delta_{i,j}\delta_{k,r}\delta_{r,s},\quad\forall\,\mathbf{p},\mathbf{r}\!\in\!(\ZZ^{\ge0})^k,\,s\!\in\!\ZZ^{\ge 0}.
\end{equation}
Define
$\dot{\cC}^{(r)}_{k,i;s,j}\!:=\!\dot{\cC}^{(r)}_{k,i;s,j}(\dot{\cY}_{n;\a})\!\in\!\QQ[\alpha][[q]]$
for $k,s$ and $r\!\in\!\ZZ^{\ge 0}$ with $s\!\le\!k\!-\!r$ and $r\!\le\!k$
by
\begin{equation}\label{eqtiC_e}
\sum_{t=0}^r\sum_{\begin{subarray}{c}
s\in\ZZ^{\ge 0}\\s\le k-t\end{subarray}}
\dot{\cC}^{(t)}_{k,i;s,j}\cC_{s,j,r+r_1-t}^{(r_1,j_1)}=\delta_{i,j_1}\delta_{k,r_1}\delta_{r,0},\qquad
\forall\,\mathbf{r}\!\in\!\ZZ^{\ge 0},\,r_1\!\le\!k\!-\!r.
\end{equation}
These equations indeed uniquely determine $\dot{\cC}^{(r)}_{k,i;s,j}$,
since
\begin{eqnarray}\label{eqtiC_e2}
\sum_{t=0}^r\sum_{\begin{subarray}{c}
s\in\ZZ^{\ge 0}\\s\le k-t\end{subarray}}
\dot{\cC}^{(t)}_{k,i;s,j}\cC_{s,j,r+r_1-t}^{(r_1,j_1)}&=&
\sum_{t=0}^{r-1}\sum_{\begin{subarray}{c}
s\in\ZZ^{\ge 0}\\s\le k-t\end{subarray}}
\dot{\cC}^{(t)}_{k,i;s,j}\cC_{s,j,r+r_1-t}^{(r_1,j_1)}\\
&&+
\sum_{\begin{subarray}{c}s\in\ZZ^{\ge 0}\\
s<r_1\end{subarray}}\dot{\cC}^{(r)}_{k,i;s,j}\cC_{s,j,r+r_1-t}^{(r_1,j_1)}
+\dot{\cC}^{(r)}_{k,i;r_1,i}.\nonumber
\end{eqnarray}
Induction by (\ref{EP_Dpq;0_e}) and (\ref{eqtiC_e2}), we obtain
\begin{equation}
\dot{\cC}^{(0)}_{k,i;s,j}=\delta_{k,s}\delta_{i,j}.
\end{equation}

Similarly we can define $\ddot{\cC}^{(r)}_{k,i;s,j}$ for $\ddot{\cY}_{n;\a}(\mathbf{x},\hbar,q)$.
\begin{prop}\label{cIrec_lmm2'}
The power series $\cD^{k,j}\dot{\cY}_{n;\a}(\mathbf{x},\hbar,q)$ and $\cD^{k,j}\ddot{\cY}_{n;\a}(\mathbf{x},\hbar,q)$ are $C$-recursive as in (\ref{recur_dfn}) with $\dot{C}_{ij}^{ik}$ and $\ddot{C}_{ij}^{ik}$ as (\ref{C}) and (\ref{C_1}). Moreover, $(\dot{\cY}_{n;\a}(\mathbf{x},\hbar,q),\cD^{k,j}\dot{\cY}_{n;\a}(\mathbf{x},\hbar,q))$ satisfies $\eta$-MPC with $\eta=\langle\a\rangle(\mathbf{x}_1+\mathbf{x}_2)^{\ell}$, and $(\dot{\cY}_{n;\a}(\mathbf{x},\hbar,q),\cD^{k,j}\ddot{\cY}_{n;\a}(\mathbf{x},\hbar,q))$ satisfies $1$-MPC.
\end{prop}
\begin{proof}
 By \cite[Lemma 5.8 (a)]{Popa}, $\gamma_j^k(\fD)\dot{\cK}$ and $\gamma_j^k(\fD)\ddot{\cK}$ are $\fC$-recursive. Then by Proposition \ref{cFrec_lmm2} and \cite[Lemma 5.8(b)]{Popa}, $\cD^{k,j}\dot{\cY}_{n;\a}(\mathbf{x},\hbar,q)$ and $\cD^{k,j}\ddot{\cY}_{n;\a}(\mathbf{x},\hbar,q)$ are $C$-recursive as in Definition \ref{recur_dfn}.

The second assertion follows from Proposition \ref{Mut}.
\end{proof}

\section{localization formula}

The standard $\TT$-action on $\CC^n$ (as well as any other representation) induces corresponding
$\TT$-actions on $\overline{Q}_{0,2}(\Gr,d)$, $\cS$, $\cV_{n;\a}^{(d)}$, $\dot\cV_{n;\a}^{(d)}$,
and $\ddot\cV_{n;\a}^{(d)}$. Thus, $\cV_{n;\a}^{(d)}$,  $\dot\cV_{n;\a}^{(d)}$, and $\ddot\cV_{n;\a}^{(d)}$ have well-defined equivariant Euler classes
$$\mathbf{e}(\cV_{n;\a}^{(d)}),\quad \mathbf{e}(\dot\cV_{n;\a}^{(d)}), \quad \mathbf{e}(\ddot\cV_{n;\a}^{(d)})
\in H_{\TT}^*\big(\overline{Q}_{0,2}(\Gr,d)\big).$$
The universal cotangent line bundle for the $i$-th marked point
also has a well-defined equivariant Euler class, which will still be
denoted by~$\psi_i$. Let
\beq\label{cZ2}
\dot{\cZ}_{n;\a}(\mathbf{x},\hbar,q) :=
1+\sum_{d=1}^{\infty}q^d \ev_{1*}\!\!\left[\frac{\mathbf{e}(\dot\cV_{n;\a}^{(d)})}{\hbar\!-\!\psi_1}\right]
\in H_{\TT}^*(\Gr)[\hbar^{-1}]\big[\big[q\big]\big],\eeq

\beq\label{ZZdfn_e2}
\ddot{\cZ}_{n;\a}(\mathbf{x},\hbar,q) :=
1+\sum_{d=1}^{\infty}q^d \ev_{1*}\!\!\left[\frac{\mathbf{e}(\ddot\cV_{n;\a}^{(d)})}{\hbar\!-\!\psi_1}\right]
\in H_{\TT}^*(\Gr)[\hbar^{-1}]\big[\big[q\big]\big],\eeq
where $\ev_1:\overline{Q}_{0,2}(\Gr,d)\lra \Gr$ is as before.

Let $[\overline{\Delta}]$ be the equivariant class of the diagonal
$$
[\overline{\Delta}]=\sum_{k,j;s,i}g_{k,j;s,i}\mathbf{\gamma}^k_j(\mathbf{x})\otimes\mathbf{\gamma}^{s}_i(\mathbf{x})\in H_{\TT}^*(\Gr)\otimes H_{\TT}^*(\Gr),
$$
where $g_{k,j;s,i}\in\QQ[\alpha]$ are homogeneous polynomials.

Let
\beq\label{Zgadfn_e2}\begin{split}
\dot{\cZ}_{\gamma^k_j}(\hbar,q) &:= \gamma^k_j(\mathbf{x})+
\sum_{d=1}^{\infty}\!q^d \ev_{1*}\!\!\left[\frac{\mathbf{e}(\dot\cV_{n;\a}^{(d)})\ev_2^*\gamma^k_j(\mathbf{x})}
{\hbar\!-\!\psi_1}\right]\in H_{\TT}^*(\Gr)[\hbar^{-1}]\big[\big[q\big]\big],\\
\ddot{\cZ}_{\gamma^k_j}(\hbar,q) &:= \gamma^k_j(\mathbf{x})+
\sum_{d=1}^{\infty}\!q^d \ev_{1*}\!\!\left[\frac{\mathbf{e}(\ddot\cV_{n;\a}^{(d)})\ev_2^*\gamma^k_j(\mathbf{x})}
{\hbar\!-\!\psi_1}\right]\in H_{\TT}^*(\Gr)[\hbar^{-1}]\big[\big[q\big]\big],
\end{split}\eeq
where $\ev_1,\ev_2\!:\overline{Q}_{0,2}(\Gr,d)\!\lra\!\Gr$.

 Let
\beq\label{ZZdfn2}
\dot{\cZ}_{n;\a}(\hbar_1,\hbar_2,q) :=
\frac{[\overline{\Delta}]}{\hbar_1+\hbar_2}+\sum_{d=1}^{\infty}q^d\big\{\ev_1\!\times\!\ev_2\}_*\!\!\left[\frac{\mathbf{e}(\dot\cV_{n;\a}^{(d)})}
{(\hbar_1\!-\!\psi_1)(\hbar_2\!-\!\psi_2)}\right]\in H_{\TT}^*(\Gr)[\hbar_1^{-1},\hbar_2^{-1}]\big[\big[q\big]\big].
\eeq
Let
\begin{eqnarray*}
\dot{\cZ}^{b_2,\omega}(\hbar,q) &:=& (-\hbar)^{b_2}\omega+
\sum_{d=1}^{\infty}\!q^d \ev_{1*}\!\!\left[\frac{\mathbf{e}(\dot\cV_{n;\a}^{(d)})\psi_2^{b_2}\ev_2^*\omega}
{\hbar\!-\!\psi_1}\right]\in H_{\TT}^*(\Gr)[\hbar^{-1}]\big[\big[q\big]\big],\\
\ddot{\cZ}^{b_2,\omega}(\hbar,q) &:=& (-\hbar)^{b_2}\omega+
\sum_{d=1}^{\infty}\!q^d \ev_{1*}\!\!\left[\frac{\mathbf{e}(\ddot\cV_{n;\a}^{(d)})\psi_2^{b_2}\ev_2^*\omega}
{\hbar\!-\!\psi_1}\right]\in H_{\TT}^*(\Gr)[\hbar^{-1}]\big[\big[q\big]\big],
\end{eqnarray*}
where $\omega\in H_{\TT}^*(\Gr)$.

Marian, Oprea and Pandharipande \cite[Section 7.3]{MOP} described the fixed loci of the $\TT$-action on $\overline{Q}_{g,m}(\Gr,d)$, which
are indexed by connected decorated graphs.
However, in
the case $g=0,\, m\!=\!2$, the relevant graphs consist of a single strand (possibly consisting
of a single vertex)
with the two marked points attached at the opposite ends of the strand.
Such a graph can be described by an ordered set $(\mathrm{Ver},<)$ of vertices,
where $<$ is a strict order on the finite set~$\mathrm{Ver}$.
Given such a strand, denote by $v_{\min}$ and $v_{\max}$ its minimal and maximal elements
and by
$\text{Edg}$ its set of edges,
i.e.~of pairs of consecutive elements.
A decorated strand is a tuple
\beq \Gamma = \big(\mathrm{Ver},<;\mu,\mathfrak{d}\big),\eeq
where $(\mathrm{Ver},<)$ is a strand as above and
$$\mu\!:\mathrm{Ver}\lra \mathbf{n}^2 \qquad\hbox{and}\qquad \mathfrak{d}\!: \mathrm{Ver}\!\sqcup\!\text{Edg}\lra(\ZZ^{\ge0})^2$$
are maps such that
\beq
\mu(v_1)\neq\mu(v_2)  ~~~\hbox{if}~~ \{v_1,v_2\}\in\text{Edg}, \qquad
\mathfrak{d}(e)\neq0~~\forall\,e\!\in\!\text{Edg}.\eeq

\begin{figure}
\begin{center}
\begin{tikzpicture}
\filldraw [gray] (1,0) circle (2pt)
(1,0) circle (2pt)
(2,0) circle (2pt)
(3,0) circle (2pt)
(4,0) circle (2pt);
\draw (0,1)--(1,0);
\draw (1,0)--(2,0);
\draw[dashed] (2,0) -- (3,0);
\draw (3,0) -- (4,0);
\draw (5,1)--(4,0);
\end{tikzpicture}
\caption{$\Gamma$}
\end{center}
\end{figure}
In Figure 1, the vertices of a decorated strand $\Gamma$
are indicated by dots in the increasing order, with respect to~$<$, from left to right.
Let $Q_{\Gamma}$ be the fixed locus of $\overline{Q}_{0,2}(\Gr,d)$ corresponding to a decorated strand $\Gamma$, then $Q_{\Gamma}$ isomorphic to
$$
\prod_{v\in \Gamma}\overline{M}_{0,2|\mathfrak{d}(v)}
$$
up to a finite quotient, where $\overline{M}_{0,2|\mathfrak{d}(v)}$ is the Hassett moduli space of weighted pointed stable curves. Let $\pi_1:\cC_{\overline{M}_{0,2|\mathfrak{d}(v)}}\to \overline{M}_{0,2|\mathfrak{d}(v)}$ be the restriction of universal family $\pi:\cC\to \overline{Q}_{0,2}(\Gr,d)$. The universal bundle is split when restricting on $\cC_{\overline{M}_{0,2|\mathfrak{d}(v)}}$. We denote by $\cS^{\vee}|_{\cC_{\overline{M}_{0,2|\mathfrak{d}(v)}}}\cong\sL_1\oplus\sL_2$.

\begin{lemm}\label{recur}$\dot{\cZ}_{n;\a}(\mathbf{x},\hbar,q), \dot{\cZ}^{b_2,\omega}(\hbar,q)$ are $\dot{C}$-recursive. Moreover, $\dot{\cZ}_{n;\a}(\mathbf{x},\hbar,q)$ satisfies $\eta$-SPC, and  $(\dot{\cZ}_{n;\a}(\mathbf{x},\hbar,q),\dot{\cZ}^{b_2,\omega}(\hbar,q))$ satisfy $\eta$-MPC with $\eta=\langle\a\rangle(\mathbf{x}_1+\mathbf{x}_2)^{\ell}$. $\ddot{\cZ}_{n;\a}(\mathbf{x},\hbar,q),\ddot{\cZ}^{b_2,\omega}(\hbar,q)$ are $\ddot{C}$-recursive, and $(\dot{\cZ}_{n;\a}(\mathbf{x},\hbar,q),\ddot{\cZ})$, $(\dot{\cZ}_{n;\a}(\mathbf{x},\hbar,q),\ddot{\cZ}^{b_2,\omega}(\hbar,q))$ satisfy $1$-MPC.
\end{lemm}
\begin{proof}

The recursive of $\dot{\cZ}_{n;\a}(\mathbf{x},\hbar,q), \ddot{\cZ}_{n;\a}(\mathbf{x},\hbar,q),\dot{\cZ}^{b_2,\omega}(\hbar,q)$ and $\ddot{\cZ}^{b_2,\omega}(\hbar,q)$ can be proved parallel to \cite[Lemma 6.5]{Zinger1}, or be obtained by just applying \cite[Lemma 7.5.2]{FK3} to the twisted case of $W\sslash G$, and $\Gr$ be considered as GIT qoutient $\CC^{2n}\sslash GL_2(\CC)$.


The proofs of polynomiality are parallel to \cite[Lemma 6.6]{Zinger1}. We sketch it as follows.

Denote by the standard action $\CC^*$ on $\CC$ by $\hbar$. Let $V=\CC\oplus\CC$ be the representation of $\CC^*$ with weight $0$, and $-\hbar$. For $d\in\ZZ^{\ge0}$, the action $\tilde{\TT}:=(\CC^*)^{\frac{n(n-1)}{2}}\times\CC^*$ on $\CC^{(\frac{n(n-1)}{2})}\otimes\text{Sym}^d V^*$ induced an action on
$$
\overline{\fX}_d=\PP(\CC^{(\frac{n(n-1)}{2})}\otimes\text{Sym}^d V^*).
$$
Let
$$
\Omega=e(\gamma^*)\in H^*_{\tilde{\TT}}(\overline{\fX}_d)
$$
be the equivariant hyperplane class.
Let
\beq
\fX_d'(\Gr)=\big\{b\!\in\!\overline{Q}_{0,2}\big(\PP V\!\times\!\Gr,(1,d)\big)\!:
\ev_1(b)\!\in\!q_1\!\times\!\Gr,~\ev_2(b)\!\in\!q_2\!\times\!\Gr\big\}.\eeq
Using Pl\"{u}cker embedding, we have the following morphism $\theta$
$$
\theta:\fX_d'(\Gr)\to\fX_d'(\PP(\CC^{(\frac{n(n-1)}{2})}))\to\overline{\fX}_d.
$$
Then applying localization formula we can prove
\begin{eqnarray*}
\Phi^{\eta}_{\dot\cZ_{n;\a}(\mathbf{x},\hbar,q),\dot{\cZ}_{n;\a}(\mathbf{x},\hbar,q)}(\hbar,z,q)&=&
\sum_{d=0}^{\infty}q^d\!\!
\int_{\fX_d'(\Gr)}\!\!\!\e^{(\theta^*\Omega)z}\mathbf{e}(\cV_{n;\a}^{(d)}),\\
\Phi^{\eta}_{\dot\cZ_{n;\a}(\mathbf{x},\hbar,q),\dot{\cZ}^{(b_2,\omega)}}(\hbar,z,q)&=&
\sum_{d=0}^{\infty}q^d\!\!
\int_{\fX_d'(\Gr)}\!\!\!\e^{(\theta^*\Omega)z}\mathbf{e}(\cV_{n;\a}^{(d)})\,
\psi_2^{b_2}\ev_2^*\omega,
\end{eqnarray*}
with $\eta=\langle\a\rangle(x_1+x_2)^{\ell}$,
and
\beq
\Phi^{1}_{\dot\cZ_{n;\a}(\mathbf{x},\hbar,q),\ddot{\cZ}^{(b_2,\omega)}}(\hbar,z,q)=\sum_{d=0}^{\infty}q^d\!\!
\int_{\fX_d'(\Gr)}\!\!\!\e^{(\theta^*\Omega)z}\mathbf{e}(\dot{\cV}_{n;\a}^{(d)})\,
\psi_2^{b_2}\ev_2^*\omega,
\eeq
which are in $\QQ_{\alpha}[\hbar][[z,q]]$.
\end{proof}

Denote by 

\[ \dot{\cI}_{n;\a}(q):= =
  \begin{cases}
    1      & \quad \text{if } |\a|<n\\
    \dot{\cY}_{n;\a}(0,0,1,q)  & \quad \text{if } |\a|=n\\
  \end{cases}
\]

 $\ddot{\cI}_{n;\a}(q):=1$.

\begin{theo}\label{thm1'}
If $\ell\!\in\!\ZZ^{\ge0}$, $n\!\in\!\ZZ^+$, and $\a\!\in\!(\ZZ^{>0})^{\ell}$ are such that
$|\a|\!\le\!n$, then
\beq \dot{\cZ}_{n;\a}(\mathbf{x},\hbar,q)=\frac{\dot{\cY}_{n;\a}(\mathbf{x},\hbar,q)}{\dot{\cI}_{n;\a}(q)}
\in H_{\TT}^*(\Gr)\big[\big[\hbar^{-1},q\big]\big],\eeq
and
\beq \ddot{\cZ}_{n;\a}(\mathbf{x},\hbar,q)=\frac{\ddot{\cY}_{n;\a}(\mathbf{x},\hbar,q)}{\ddot{\cI}_{n;\a}(q)}
\in H_{\TT}^*(\Gr)\big[\big[\hbar^{-1},q\big]\big].\eeq
\end{theo}
\begin{proof} The proof is parallel to \cite[Theorem 3]{CZ}. 

When $|\a|\le n-2$, 
$$
\dot{\cY}_{n;\a}(\mathbf{x},\hbar,q)=1 \quad (\mathrm{mod}\ \hbar^{-2}).
$$ 
For $d\in\ZZ^{+}$,
$$
\dim Q_{0,2}(\Gr,d)-\mathrm{rk}\dot{\cV}_{n;\a}^{(d)}=(n-|a|)d+2(n-2)-1>2(n-2)=\dim \Gr.
$$ 
Thus 
$$
\dot{\cZ}_{n;\a}(\mathbf{x},\hbar,q)=1 \quad (\mathrm{mod}\ \hbar^{-2}).
$$
Therefore 
$$
\dot{\cZ}_{n;\a}(\mathbf{x},\hbar,q)=\dot{\cY}_{n;\a}(\mathbf{x},\hbar,q).
$$
When $n-1\le|\a|\le n$.
  By Lemma \ref{recur}, $\dot{\cZ}(\hbar,q)$ is $C$-recursive, therefore
\begin{eqnarray}
\label{recurdfn_e2'} &&\dot{\cZ}_{n;\a}(\mathbf{x},\hbar,q)(\alpha_i,\alpha_j,\hb,q)\\
&=&\sum_{d=0}^{\infty}\sum_{r=-N_d}^{N_d}
\dot{\cZ}_{ij}^r(d)\hb^rq^d+
\sum_{d=1}^{\infty}\sum_{k\neq j}\frac{\dot{C}_{ij}^{ik}(d)q^d}{\hb-\frac{\alpha_k-\alpha_j}{d}}
\cZ(\alpha_i,\alpha_k,(\alpha_k\!-\!\alpha_j)/d,q)\nonumber\\
&&+
\sum_{d=1}^{\infty}\sum_{k\neq i}\frac{\dot{C}_{ij}^{kj}(d)q^d}{\hb-\frac{\alpha_k-\alpha_i}{d}}
\dot{\cZ}(\alpha_k,\alpha_j,(\alpha_k\!-\!\alpha_i)/d,q), \qquad\forall~ij,ik,kj\!\in\!\mathbf{n},\nonumber
\end{eqnarray}
for some $\dot{\cZ}_{ij}^r(d)\!\in\!\QQ_{\alpha}$.

Parallel to the proof of \cite[Proposition 6.1]{CZ}, by applying localization formula to $\dot{\cZ}(\hbar,q)$, for all $r<0$, we have
\begin{eqnarray}
&&\sum_{d=1}^{\infty}\dot{\cZ}_{ij}^r(d)q^d\\
&=&\sum_{d=1}^{\infty}\frac{q^d}{d!}\bigg(\big(\sum_{d_1+d_2=d}\int_{\overline{\cM}_{0,2|(d_1,d_2)}}
\frac{\prod_{k\neq i,j}(\alpha_i-\alpha_k)\prod_{k\neq i, j}(\alpha_j-\alpha_k)\mathbf{e}(\dot\cV_{n;\a}^{(d)}(\alpha_i+\alpha_j))\psi_1^{-r-1}\psi_2^b}{\prod_{k\neq i,j} \mathbf{e}(R^0\pi_{1*}\sL_1(\alpha_i-\alpha_k))\prod_{k\neq i,j}\mathbf{e}(R^0\pi_{1*}\sL_2(\alpha_j-\alpha_k))}\big)\nonumber\\
&&\times\mathrm{Res}_{\hbar=0}\big\{\frac{(-1)^b}{\hbar^{b+1}}\dot{\cZ}(\alpha_i,\alpha_j,\hbar,q)\big\}\bigg)\nonumber.
\end{eqnarray}

Parallel to the proof of \cite[Proposition 8.3]{CZ}, we have
\begin{eqnarray}
&&\mathrm{Res}_{\hbar=0}\{\hbar^{-r}\dot{\cY}(\alpha_i,\alpha_j,\hbar,q)\}\\
&=&\sum_{d=1}^{\infty}\frac{q^d}{d!}\bigg(\big(\sum_{d_1+d_2=d}\int_{\overline{\cM}_{0,2|(d_1,d_2)}}
\frac{\prod_{k\neq i,j}(\alpha_i-\alpha_k)\prod_{k\neq i, j}(\alpha_j-\alpha_k)\mathbf{e}(\dot\cV_{n;\a}^{(d)}(\alpha_i+\alpha_j))\psi_1^{-r-1}\psi_2^b}{\prod_{k\neq i,j} \mathbf{e}(R^0\pi_{1*}\sL_1(\alpha_i-\alpha_k))\prod_{k\neq i,j}\mathbf{e}(R^0\pi_{1*}\sL_2(\alpha_j-\alpha_k))}\big)\nonumber\\
&&\times\mathrm{Res}_{\hbar=0}\big\{\frac{(-1)^b}{\hbar^{b+1}}\dot{\cY}(\alpha_i,\alpha_j,\hbar,q)\big\}\bigg)\nonumber.
\end{eqnarray}
This means that $\dot{\cZ}_{ij}^r(d)$ and $\frac{\dot{\cY}_{ij}^r(d)}{\dot{\cI}}$ satisfy the same recursive relations. For their coefficients of $q^0$ are the same,
$$
\sum_{d=1}^{\infty}\dot{\cZ}_{ij}^r(d)q^d=\frac{1}{\dot{\cI}_{n;\a}(q)}\sum_{d=1}^{\infty}\dot{\cY}_{ij}^r(d)q^d,
$$
$$
\dot{\cZ}_{n;\a}(\alpha_i,\alpha_j,\hb,q)= \frac{\dot{\cY}_{n;\a}(\alpha_i,\alpha_j,\hbar,q)}{\dot{\cI}(q)}\quad (\mathrm{mod}\ \hbar^{-2}).
$$
Then by \cite[Lemma 30.3.2]{MirSym},
$$
\dot{\cZ}_{n;\a}(\mathbf{x},\hbar,q)=\frac{\dot{\cY}_{n;\a}(\mathbf{x},\hbar,q)}{\dot{\cI}_{n;\a}(q)}.
$$
Because $(\dot{\cY}_{n;\a}(\mathbf{x},\hbar,q),\ddot{\cY}_{n;\a}(\mathbf{x},\hbar,q))$ satisfy 1-MPC, $(\dot{\cZ}_{n;\a}(\mathbf{x},\hbar,q)=\frac{\dot{\cY}_{n;\a}(\mathbf{x},\hbar,q)}{\dot{\cI}_{n;\a}(q)},\ddot{\cY})$ satisfy 1-MPC. For $(\dot{\cZ}_{n;\a}(\mathbf{x},\hbar,q),\ddot{\cZ}_{n;\a}(\mathbf{x},\hbar,q))$ satisfy 1-MPC, and 
$$
\ddot{\cZ}_{n;\a}(\alpha_i,\alpha_j,\hbar,q)=\frac{\ddot{\cY}_{n;\a}(\alpha_i,\alpha_j,\hbar,q)}{\ddot{\cI}_{n;\a}(q)}\quad (\mathrm{mod}\ \hbar^{-1}).
$$
By Proposition \ref{uniqueness_prp}, we have
$$
\ddot{\cZ}_{n;\a}(\mathbf{x},\hbar,q)=\frac{\ddot{\cY}_{n;\a}(\mathbf{x},\hbar,q)}{\ddot{\cI}_{n;\a}(q)}.
$$
\end{proof}

\begin{theo}\label{thm2'}
If $\ell\!\in\!\ZZ^{\ge0}$, $n\!\in\!\ZZ^+$, and $\a\!\in\!(\ZZ^{>0})^{\ell}$ are such that
$|\a|\!\le\!n$, then
\begin{eqnarray*}
\dot{\cZ}_{\gamma^k_j}(\hbar,q)&=&\dot{\cY}_{\gamma^k_j}(\hbar,q)\!\in\!H_{\TT}^*(\Gr)[\hbar^{-1}][[q]],\\
\ddot{\cZ}_{\gamma^k_j}(\hbar,q)&=&\ddot{\cY}_{\gamma^k_j}(\hbar,q)\!\in\!H_{\TT}^*(\Gr)[\hbar^{-1}][[q]],
\end{eqnarray*}
where
\begin{eqnarray*}
\dot{\cY}_{\gamma^k_j}(\hbar,q)&:=&\cD^{k,j}\dot{\cY}_{n;\a}(\mathbf{x},\hbar,q)+\sum_{t=1}^{k}\sum_{s=0}^{k-t}\dot{\cC}^{(t)}_{k,j;s,i}(q)\hbar^{k-t-s}
\cD^{s,i}\dot{\cY}_{n;\a}(\mathbf{x},\hbar,q)\!\in\!\QQ(x,\hbar)[[q]],\\
\ddot{\cY}_{\gamma^k_j}(\hbar,q)&:=&\cD^{k,j}\ddot{\cY}_{n;\a}(\mathbf{x},\hbar,q)+\sum_{t=1}^{k}\sum_{s=0}^{k-t}\ddot{\cC}^{(t)}_{k,j;s,i}(q)\hbar^{k-t-s}
\cD^{s,i}\ddot{\cY}_{n;\a}(\mathbf{x},\hbar,q)\!\in\!\QQ(x,\hbar)[[q]],
\end{eqnarray*}
and $\dot{\cC}^{(r)}_{k,i;s,j}, \ddot{\cC}^{(r)}_{k,i;s,j}$ are defined as in (\ref{eqtiC_e}).
\end{theo}
\begin{proof}
By the definition of $\dot{\cC}^{(r)}_{k,i;s,j}$,
$$
\dot{\cY}_{\gamma^k_j}(\hbar,q)=\gamma^k_j(\mathbf{x}) = \dot{\cZ}_{\gamma^k_j}(\hbar,q) \quad (\mathrm{mod}\ \hbar^{-1}).
$$
By Proposition \ref{uniqueness_prp} and Lemma \ref{recur},
$$
\dot{\cY}_{\gamma^k_j}(\hbar,q)=\dot{\cZ}_{\gamma^k_j}(\hbar,q).
$$
Similar argument holds for $\ddot{\cZ}_{\gamma^k_j}(\hbar,q)$ and $\ddot{\cY}_{\gamma^k_j}(\hbar,q)$.
\end{proof}

\begin{theo}\label{main1}If $\ell\!\in\!\ZZ^{\ge0}$, $n\!\in\!\ZZ^+$, and $\a\!\in\!(\ZZ^{>0})^{\ell}$ are such that
$|\a|\!\le\!n$, then
\beq\label{main_1}
\dot{\cZ}_{n;\a}(\hbar_1,\hbar_2,q) =\frac{1}{\hbar_1+\hbar_2}\sum_{k,j;s,i}g_{k,j;s,i}\dot{\cY}_{\gamma^k_j}(\hbar_1,q)\ddot{\cY}_{\gamma^{s}_i}(\hbar_2,q).
\eeq
\end{theo}
\begin{proof}By (\ref{zero}) and (\ref{AB}),
\beq
(\hbar_1+\hbar_2)\dot{\cZ}_{n;\a}(\hbar_1,\hbar_2,q)|_{p_{ij}\times p_{kl}}=[\overline{\Delta}]|_{p_{ij}\times p_{kl}}+(\hbar_1+\hbar_2)\sum_{d>0} q^d\int_{\overline{Q}_{0,2}(\Gr,d)}\frac{\mathbf{e}(\dot\cV_{n;\a}^{(d)})\ev_1\phi_{ij}\ev_2\phi_{kl}}{(\hbar_1-\psi_1)(\hbar_2-\psi_2)},
\eeq
for all $ij,kl\in \mathbf{n}^2$.
By Lemma \ref{recur}, the coefficient of $\hbar_2^m$ in $ (\hbar_1+\hbar_2)\dot{\cZ}_{n;\a}(\hbar_1,\hbar_2,q)$ is $\dot{C}$-recursive and satisfies $\eta=\langle\a\rangle(\mathbf{x}_1+\mathbf{x}_2)^{\ell}$-MPC with respect to $\dot{\cZ}_{n;\a}(\hbar,q)$ for all $m\le0$. By Proposition \ref{uniqueness_prp}, it follows that in order to prove (\ref{main_1}) it suffices to show that
\beq\label{h-com}
(\hbar_1+\hbar_2)\dot{\cZ}_{n;\a}(\hbar_1,\hbar_2,q)|_{p_{ij}\times p_{kl}} =\sum_{k,j;s,i}g_{k,j;s,i}\dot{\cY}_{\gamma^k_j}(\hbar_1,q)\ddot{\cY}_{\gamma^{s}_i}(\hbar_2,q)|_{p_{ij}\times p_{kl}}\quad (\mathrm{mod}\ \hbar_1^{-1})
\eeq
The right-hand side of (\ref{h-com}) mod $\hbar^{-1}_1$ is
\beq\label{h_1}
\sum_{k,j;s,i}g_{k,j;s,i}\gamma^k_j|_{ p_{kl}}\ddot{\cY}_{\gamma^{s}_i}(\hbar_2,q)|_{ p_{kl}}.
\eeq
The left-hand side of (\ref{h-com}) mod $\hbar^{-1}_1$ is
\beq\label{h_2}
[\overline{\Delta}]|_{p_{ij}\times p_{kl}}+\sum_{d>0} q^d\int_{\overline{Q}_{0,2}(\Gr,d)}\frac{\mathbf{e}(\dot\cV_{n;\a}^{(d)})\ev_1\phi_{ij}\ev_2\phi_{kl}}{\hbar_2-\psi_2}.
\eeq
By Lemma \ref{recur}, (\ref{h_2}) is $\ddot{C}$-recursive and satisfies $\eta=1$-MPC with respect to $\dot{\cZ}_{n;\a}(\hbar_2,q)$. Since (\ref{h_1}) also satisfies these properties. The power series
(\ref{h_2}) and (\ref{h_1}) agree if only if they agree mod $\hbar_2^{-1}$. The latter is the case since both are the equivariant Poincar\'{e} dual to the diagonal restricted to $p_{ij}\times p_{kl}$.
\end{proof}

Denote by 
$$\dot{C}^{(t)}_{k,j;s,i}(q)=\dot{\cC}^{(t)}_{k,j;s,i}(q)|_{\alpha=0}, \quad
\ddot{C}^{(t)}_{k,j;s,i}(q)=\ddot{\cC}^{(t)}_{k,j;s,i}(q)|_{\alpha=0}, \quad
\dot{I}=\dot{\cI}|_{\alpha=0}, \quad
\ddot{I}_{n;\a}=\ddot{\cI}_{n;\a}|_{\alpha=0},$$
$$D^{k,j}\dot{Y}_{n;\a}(\hbar,q)=\cD^{k,j}\dot{\cY}_{n;\a}(\hbar,q)|_{\alpha=0}, \quad
D^{k,j}\ddot{Y}_{n;\a}(\hbar,q)=\cD^{k,j}\ddot{\cY}_{n;\a}(\hbar,q)|_{\alpha=0}.$$
By restricting on $\alpha=0$, Theorem \ref{thm1'}, Theorem \ref{thm2'} and Theorem \ref{main1} give us the non-equivariant Theorem \ref{thm1}, Theorem \ref{thm2} and Theorem \ref{main1'}.

\end{document}